\documentclass[12pt]{article}
\usepackage[a4paper,width=6.8in,height=8.8in]{geometry}
\usepackage{amsmath,amssymb}

\usepackage{amsmath,amssymb,mathtools,amsthm,
	textcomp,mathscinet}
\usepackage{amsfonts,graphicx,enumerate,subcaption,placeins}
\usepackage[hidelinks]{hyperref}
\usepackage{setspace}
\usepackage[utf8]{inputenc}
\usepackage[english]{babel}
\usepackage{booktabs,multirow}
\usepackage{blkarray}
\usepackage{lscape}
\usepackage[T1]{fontenc}
\usepackage{authblk,color}
\usepackage[english]{babel}
\usepackage{amsfonts, amstext, amsthm, booktabs, dcolumn}
\usepackage{graphicx}
\usepackage{float}
\usepackage{caption}
\usepackage{centernot}
\usepackage{times}
\usepackage[nottoc]{tocbibind}
\usepackage[mathscr]{eucal}

\definecolor{gr}{rgb}{0.7, 0.0, 0.15}

\newtheorem{theorem}{\bf Theorem}[section]

\newtheorem{corollary}{\bf Corollary}[section]
\newtheorem{remark}{\bf Remark}[section]
\newtheorem{remarks}{\bf Remarks}[section]

\newtheorem{example}{\bf Example}[section]

\newcommand\floor[1]{\left\lfloor#1\right\rfloor}
\newcommand\ceil[1]{\left\lceil#1\right\rceil}

\numberwithin{equation}{section}
\begin{document}
\title{Approximations to Weighted Sums of Random Variables}
\author[ ]{{\Large Amit N. Kumar}}
\affil[ ]{Department of Mathematics}
\affil[ ]{Indian Institute of Technology Bombay}
\affil[ ]{Powai, Mumbai-400076, India.}
\affil[ ]{Email: amit.kumar2703@gmail.com}
\date{}
\maketitle

\begin{abstract}
\noindent
In this paper, we obtain error bound for pseudo-binomial and negative binomial approximations to weighted sums of locally dependent random variables, using Stein's method. We also discuss approximation results for weighted sums of independent random variables. We demonstrate our results through some applications in finance and runs in statistics.
\end{abstract}

\noindent

\section{Introduction and Preliminaries}
Sums of random variables (rvs) have always a special attraction as it raises in relevant theoretical challenges. Moreover, several linear statistics can be represented as weighted sums of rvs. Also, it is difficult to find the exact distribution of weighted sums of rvs, especially, if the underlying rvs are non-identical. So, It is of interest to study the behaviour of such distributions. Many researchers studied the limiting behavior of weighted sums of rvs such as Chow and Lai \cite{CL1973}, Olvera-Cravioto \cite{OCM2012}, and Zhengyan \cite{ZL1989}, among many others. But if weights are natural numbers, then it is also difficult to get asymptotic limits. Therefore, the study of the proximity of such distributions with a suitable distribution is of interest when the summation is taken over a finite set.\\
In this paper, we consider weighted sums of $\mathbb{Z}_+$-valued rvs, where $\mathbb{Z}_+=\{0,1,2,\ldots\}$, the set of non-negative integers, and propose its approximation with pseudo-binomial and negative binomial distributions by matching the first two moments. Also, we assume weights are natural numbers. Let $X_1\sim$ PB$(N,p)$, the pseudo-binomial distribution (see \v{C}ekanavi\v{c}ius and Roos \cite{CR}, p. 370), and $X_2\sim$ NB$(r,\bar{p})$, the negative binomial distribution, then their probability mass functions are given by
\begin{align*}
\mathbb{P}(X_1=k)=\frac{1}{\delta}\binom{N}{k}{p}^{k}q^{N-k}, \quad k=0,1,\ldots,\floor{N}
\end{align*} 
and
\begin{align*}
\mathbb{P}(X_2=k)=\binom{r+k-1}{k}\bar{p}^r \bar{q}^{k}, \quad k=0,1,\ldots,
\end{align*}
respectively, where $N>1$, $r>0$, $0<q=1-p<1$, $0<\bar{q}=1-\bar{p}<1$, $\delta=\sum_{k=0}^{\floor{N}}\binom{N}{k}{p}^{k}q^{N-k}$, $\binom{N}{k}=N(N-1)\cdots (N-k+1)/k!$, and $\floor{N}$ is the greatest integer function of $N$. The study of asymptotic behavior for weighted sums of rvs is discussed widely in the literature under certain conditions on weights, such as the sum of squares of weights is finite (see Chow and Lai \cite{CL1973}), weights are normalized (see Etemadi \cite{EN2006}) and geometrically weighted (see Bhati and Rattihalli \cite{BR2014}), among many others. However, we consider weights are natural numbers which do not satisfy these type of conditions and obtain error bounds for pseudo-binomial and negative binomial approximations. This study of proximity is useful to identify the behavior of such distributions over a finite set. We use the total variation distance metric and Stein's method to derive our approximation results.\\
Next, let ${\cal G}=\{g:{\mathbb Z}_+\to{\mathbb R}|~g~\text{is bounded}\}$ and ${\cal G}_Y=\{g\in {\cal G}|~g(0)=0~\text{and}~g(y)=0,~\text{for}~y\notin {\bf \text{S}}(Y)\}$, for a $\mathbb{Z}_+$-valued rv $Y$, where {\bf S}$(Y)$ is the support of the rv $Y$. We discuss briefly Stein's method (Stein \cite{stein1972}) which can be carried out mainly in the three steps. First, compute a Stein operator $\mathcal{A}_Y$ satisfies ${\mathbb E}[{\cal A}_Y g(Y)]=0$, for $g \in {\cal G}_Y$. Second, find the solution of Stein equation
\begin{equation}
{\cal A}_Yg(k)=f(k)-{\mathbb E}f(Y),~f \in {\cal G}~\text{and}~g \in {\cal G}_Y.\label{8:seq}
\end{equation}
Finally, use a rv $Z$ in place of $k$ in \eqref{8:seq} and take expectation and supremum which leads to the total variation distance between $Y$ and $Z$ as follows:
\begin{equation}
d_{TV}(Y,Z):=\sup_{f \in {\cal H}}|{\mathbb E}f(Y)-{\mathbb E}f(Z)|=\sup_{f \in {\cal H}}|{\mathbb E}{\cal A}_Yg(Z)|,\label{8:power}
\end{equation}
where ${\cal H}=\{I(A)| A \subseteq {\mathbb Z}_+\}$ and $I(A)$ is the indicator function of the set $A$.

\noindent 
Next, consider a rv $X$ and its Stein operator of the form
\begin{align}
\mathcal{A}_Xg(k)=(\alpha+\beta k)g(k+1)-kg(k), \quad k\in \mathbb{Z}_+,~g\in \mathcal{G}_X, \label{8:com}
\end{align}
which represent pseudo-binomial Stein operator if $\alpha=Np/q$ and $\beta=-p/q$ and negative binomial Stein operator if $\alpha=r \bar{q}$ and $\beta=\bar{q}$, respectively. For details, see (5) and (6) of Upadhye {\em et al.} \cite{UCV}. Also, the upper bound for the solution of \eqref{8:seq} (say $g_f$) is given by
\begin{align}
\|\Delta g_f\|\le \left\{
\begin{array}{ll}
1/\floor{N}p, & \text{if $X\sim$ PB$(N,p)$;}\\
1/r \bar{q}, & \text{if $X\sim$ NB$(r,\bar{p})$},
\end{array}\right.\label{8:bd}
\end{align}
where $\Delta g_f(k)=g_f(k+1)-g_f(k)$ and $\|\Delta g_f\|=\sup_k |\Delta g_f(k)|$. See (2.8), (2.10), and (2.11) of Kumar {\em et al.} \cite{KUV} and (57) of \v{C}ekanavi\v{c}ius and Roos \cite{CR} for more details. Observe that
\begin{align*}
\frac{\alpha}{1-\beta}=\left\{
\begin{array}{ll}
N p, & \text{if $X\sim$ PB$(N,p)$};\\
r \bar{q}/\bar{p}, & \text{if $X\sim$ NB$(r,\bar{p})$}
\end{array}\right.
\quad \text{and}\quad
\frac{\alpha}{(1-\beta)^2}=\left\{
\begin{array}{ll}
N pq, & \text{if $X\sim\text{PB}(N,p)$;}\\
r \bar{q}/\bar{p}^2, & \text{if $X\sim\text{NB}(r,\bar{p})$}
\end{array}\right.
\end{align*}
are mean and variance of pseudo-binomial and negative binomial distributions, respectively. For more details, we refer the reader to Brown and Xia \cite{BX}, Eichelsbacher and Reinert \cite{ER}, Kumar {\em et al.} \cite{KUV}, Ley {\em et al.} \cite{LRS}, Upadhye and Barman \cite{NK}, Upadhye {\em et al.} \cite{UCV}, and references therein.\\
This paper is organized as follows. In Section \ref{8:MR}, we present our main results and discuss some relevant remarks and applications. In Section \ref{8:pr}, we give the proofs of our main results.

\section{Main Results}\label{8:MR}
Let $J\subset \mathbb{N}=\{1,2,\ldots\}$ be finite and $\{\eta_i,i\in J\}$ be a collection of $\mathbb{Z}_+$-valued random variables. Also, for each $i$, let $w_i\in \mathbb{N}$, $\mathbb{E}(\omega_i \eta_i)^3<\infty$, and $i\in A_i\subseteq B_i\subset J$ be such that $\eta_i$ is independent of $\eta_{A_i^c}$ and $\eta_{A_i}$ is independent of $\eta_{B_i^c}$, where $\eta_A$ is the collection of random variables $\{\eta_i,i\in A\}$ and $A^c$ denotes the complement of the set $A$. See Section 3 of R\"{o}llin \cite{RO2008} for a similar type of locally dependent structure. In addition, if $A_i=B_i=\{i\}$ then our locally dependent structure reduced to the independent collection of random variables. Now, let $w_i=1$ for at least one $i\in J$ and define
\begin{align}
W:=\sum_{i\in J}\omega_i \eta_i,\label{8:Wn}
\end{align} 
the weighted sum of locally dependent random variables. For a set $A\subset J$, we let $\eta_{A}^{*}=\sum_{i\in A}\omega_i \eta_i$. For any random variables $Z$, we define $\mathcal{D}(Z):=2d_{TV}(\mathcal{L}(Z),\mathcal{L}(Z+1))$. Throughout this section, let $X$ be a random variable having Stein operator \eqref{8:com} and
\begin{align}
\alpha=\frac{(\mathbb{E}W)^2}{\mathrm{Var}(W)}\quad \text{and}\quad \beta=\frac{\mathrm{Var}(W)-\mathbb{E}W}{\mathrm{Var}(W)}\label{8:MM}
\end{align}
so that $\mathbb{E}X=\mathbb{E}W$ and $\mathrm{Var}(X)=\mathrm{Var}(W)$.

\subsection{Locally Dependent Random Variables}
In this subsection, we consider $\{i\}\subset A_i\subset B_i$ and discuss the approximation result for the weighted sum of locally dependent random variables.
\begin{theorem}\label{8:th1}
Let $W$ be the weighted sum of locally random variables as defined in \eqref{8:Wn} and $X$ be a random variable having Stein operator \eqref{8:com} satisfying \eqref{8:MM}. Then
\begin{align}
d_{TV}(\mathcal{L}(W),\mathcal{L}(X))&\le \|\Delta g\|\left\{ \frac{(1-\beta)}{2}\left[\sum_{i\in J} \omega_i\mathbb{E}\eta_i \mathbb{E}[\eta_{A_i}^*(2\eta_{B_i}^{*}-\eta_{A_i}^{*}-1)\mathcal{D}(W|\eta_{A_i},\eta_{B_i})]\right.\right.\nonumber\\
&~~~~~~~~\left.+\sum_{i\in J}\omega_i\mathbb{E}[\eta_i(\eta_{A_i}^{*}-1)(2\eta_{B_i}^{*}-\eta_{A_i}^{*}-2)\mathcal{D}(W|\eta_i,\eta_{A_i},\eta_{B_i})]\right]\nonumber\\
&~~~~~~~~~+\sum_{i\in J}\omega_i\bigr|(1-\beta)[\mathbb{E}(\eta_i) \mathbb{E}(\eta_{A_i}^*)-\mathbb{E}(\eta_i(\eta_{A_i}^{*}-1))]+\beta\mathbb{E}(\eta_i)\bigr|\mathbb{E}[\eta_{B_i}^{*}\mathcal{D}(W|\eta_{B_i})]\nonumber\\
&~~~~~~~~~\left.+|\beta| \sum_{i\in J} \omega_i\mathbb{E}[\eta_i(\eta_{B_i}^{*}-1)\mathcal{D}(W|\eta_{B_i})]\right\},\label{8:mbdd}
\end{align}
where the upper bound of $\|\Delta g\|$ is given in \eqref{8:bd}.
\end{theorem}

\begin{remarks}
\begin{itemize}
\item[(i)] Observe that $W$ can be represented as a conditional sum of independent random variables. Therefore, Subsections 5.3 and 5.4 of R\"{o}llin \cite{RO2008} are useful to find the upper bound of $\mathcal{D}(W|\cdot)$.
\item[(ii)] The choice of parameters in \eqref{8:MM} is valid if $\mathbb{E}W>\mathrm{Var}(W)$ and $\mathbb{E}W<\mathrm{Var}(W)$ for pseudo-binomial and negative binomial approximations, respectively.
\end{itemize}
\end{remarks}

\noindent
Next, we discuss some applications of Theorem \ref{8:th1}.

\begin{example}[{\bf (1,1)-runs}]\label{8:ex1}
Let $J=\{1,2,\ldots,n\}$ and $\{\zeta_i,i \in J\}$ be a sequence of independent Bernoulli trials with success probability $p_i=\mathbb{P}(\zeta_i=1)=1-\mathbb{P}(\zeta_i=0)$. Also, let $A_i=\{j:|j-i|\le 1\}\cap J$, $B_i=\{j:|j-i|\le 2\}\cap J$, $\bar{\zeta}_i=(1-\zeta_{i-1})\zeta_i$, and $W_n=\sum_{i=2}^{n}\bar{\zeta}_i$. Then, the distribution of $W_n$ is known as the distribution of $(1,1)$-runs and it adopted our locally dependent structure with $\omega_i=1$. For more details, see Huang and Tsai \cite{HT}, Upadhye {et al.} \cite{UCV}, Vellaisamy \cite{V}, and reference therein.\\
Next, it can be easily verified that 
\begin{align*}
\sum_{i=2}^{n}(1-p_{i-1})p_i=\mathbb{E}(W_n)>\mathrm{Var}(W_n)=\sum_{i=2}^{n}(1-p_{i-1})p_i-\sum_{i=2}^{n}\sum_{j\in A_i}(1-p_{i-1})(1-p_{j-1})p_ip_j.
\end{align*} 
Therefore, the pseudo-binomial approximation to $W_n$ is suitable in the view of the valid choice of parameters. Now, let $\mathcal{D}(W_i^\star)=\mathcal{D}(W_n|\bar{\zeta}_{B_i})$ and $\bar{\zeta}_{e}=\{\bar{\zeta}_{2k},1\le k\le\floor{n/2}\}$ then $\mathcal{L}(W_i^\star|\bar{\zeta}_e=k)$ can be represented as the sum of independent random variables (say $\zeta_i^{k}$, $i\in\{1,2,\ldots,n_k\}=:\mathcal{F}_k$ and $k\in\{0,1\}^{\floor{n/2}}$), and therefore, from (5.11) if R\"{o}llin \cite{RO2008}, we have
\begin{align*}
\mathcal{D}(W_i^\star)\le \mathbb{E}\{\mathbb{E}[W_i^\star|\bar{\zeta}_e]\}\le \mathbb{E}\left[\frac{2}{V_{\bar{\zeta}_e}^{1/2}}\right], 
\end{align*}
where $V_{\bar{\zeta}_e=k}=\sum_{j \in \mathcal{F}_k}\min\left\{1/2,1-\mathcal{D}(\zeta_j^k)\right\}$. Let $1/2 \ge \mathbb{P}(\zeta_i^k=1)=1-\mathbb{P}(\zeta_i^k=0)=\sum_{k_1,k_2\in\{0,1\}}$ $\mathbb{P}(\bar{\zeta}_i=1|\bar{\zeta}_{i-1}=k_1,\bar{\zeta}_{i+1}=k_2)=\mathbb{P}(\bar{\zeta}_i=1|\bar{\zeta}_{i-1}=0,\bar{\zeta}_{i+1}=0)=(1-p_{i-1})p_i$. Therefore,
\begin{align*}
1-\mathcal{D}(\zeta_i^k)&=1-\frac{1}{2}\sum_{m=0}^{1}|\mathbb{P}(\zeta_i^k=m-1)-\mathbb{P}(\zeta_i^k=m)|\\
&=1-\frac{1}{2}\{2\mathbb{P}(\zeta_i^k=0)-\mathbb{P}(\zeta_i^k=1)\}=\frac{3}{2}(1-p_{i-1})p_i.
\end{align*}
Hence, $V_{\bar{\zeta}_e=k}=\frac{1}{2}\sum_{j \in \mathcal{F}_k}\min\left\{1,3(1-p_{i-1})p_i\right\}$ and, from Theorem \ref{8:th1} with the pseudo-binomial setting, we have
\begin{align*}
d_{TV}(\mathcal{L}(W_n),\text{PB}(\hat{n},\hat{p}))&\hspace{-0.01cm}\le \frac{1}{\floor{\hat{n}}\hat{p}\hat{q}}\left\{\sum_{i=2}^{n}\big[ \mathbb{E}\bar{\zeta}_i \mathbb{E}[\bar{\zeta}_{A_i}^*(2\bar{\zeta}_{B_i}^{*}-\bar{\zeta}_{A_i}^{*}-1)]+\mathbb{E}[\bar{\zeta}_i(\bar{\zeta}_{A_i}^{*}-1)(2\bar{\zeta}_{B_i}^{*}-\bar{\zeta}_{A_i}^{*}-2)]\big]\right.\\
&~~~~~~~~~~~~~~~~~~~+2\sum_{i=2}^n\bigr|\mathbb{E}(\bar{\zeta}_i) \mathbb{E}(\bar{\zeta}_{A_i^*})-\mathbb{E}(\bar{\zeta}_i(\bar{\zeta}_{A_i}^{*}-1))-\hat{p}\mathbb{E}(\bar{\zeta}_i)\bigr|\mathbb{E}[\bar{\zeta}_{B_i}^{*}]\\
&~~~~~~~~~~~~~~~~~~~\left.+2\hat{p} \sum_{i=2}^{n}\mathbb{E}[\bar{\zeta}_i(\bar{\zeta}_{B_i}^{*}-1)]\right\}\left(\frac{1}{2}\sum_{j \in \mathcal{F}}\min\left\{1,3(1-p_{i-1})p_i\right\}\right)^{-1/2},
\end{align*}
where $\hat{n}=\left(\sum_{i=2}^{n}(1-p_{i-1})p_i\right)^2/\sum_{i=2}^{n}\sum_{j\in A_i}(1-p_{i-1})(1-p_{j-1})p_ip_j$, $\hat{p}=\sum_{i=2}^{n}\sum_{j\in A_i}(1-p_{i-1})(1-p_{j-1})p_ip_j/\sum_{i=2}^{n}(1-p_{i-1})p_i$, and $\mathcal{F}=\min_{k}\mathcal{F}_k$. Note that the above bound is of $O(n^{-1/2})$ which is an improvement over (77) of Upadhye {et al.} \cite{UCV}, Theorem 2.1 of Vellaisamy \cite{V}, which are of $O(1)$, and Theorem 2.1 of Godbole \cite{G1993}, which is of order $O(n)$.
\end{example}

\begin{example}[{\bf Collateralized Debt Obligation (CDO)}]\label{8:ex2}
A CDO is a type of asset-backed security that transferred pool of assets into a product and sold to investors. These assets divided into a set of repayment which is called tranches. The tranches have different payment priorities and interest rates. The basic tranches used in CDO are senior, mezzanine, and equity. Investors can invest in their interested tranches. For more details, see Neammanee and Yonghint \cite{NY}, Yonghint et al. \cite{YNC2}, and reference therein.\\
In Yonghint et al. \cite{YNC2}, it is demonstrated that the locally dependent CDO occurs in the borrower's related assets that arise from several groups. If the element of groups have weights in terms of economy. Then the weighted locally dependent CDO is also useful in applications.\\
We consider the CDO similar to discussed by Yonghint et at. \cite{YNC2}. Let the CDO tranche pricing is based on $\bar{n}$ assets and the recovery rate of $i^{\text{th}}$ assets is $R_i>0$. The percentage cumulative loss in CDO up to time $T$ is 
\begin{align*}
L(T)=\frac{1}{\bar{n}}\sum_{i=1}^{\bar{n}}(1-R_i)\omega_i I_i,
\end{align*}
where $I_i=I(\tau_i \le T)$, and $\tau_i$ is the default time of the $i^{\text{th}}$ asset. Assume the recovery rate is constant, say $R$, then the CDO pricing problem is reduced to calculate
\begin{align}
\mathbb{E}[(L(T ) - z^*)^+]=\frac{1-R}{\bar{n}}\mathbb{E}[(\overline{W}_{\bar{n}}-z)^+],\label{8:jjj}
\end{align} 
where $z^*=(1-R)z/\bar{n}>0$ is the attachment or the detachment point of the tranche, $\overline{W}_{\bar{n}}=\sum_{i=1}^{\bar{n}} \omega_i I_i$, and $(a)^+=\max(a,0)$. Note that, from \eqref{8:jjj}, it is sufficient to deal with $\mathbb{E}[(\overline{W}_{\bar{n}}-z)^+]$. For additional details, see Yonghint et al. \cite{YNC2} and reference therein.\\ 
We are interested to approximate $\mathbb{E}[(\overline{W}_{\bar{n}}-z)^+]$ by $\mathbb{E}[(\text{PB}(N,p)-z)^+]$. First, let us modify the Stein equation \eqref{8:seq} as
\begin{align}
(k-z)^+-\mathbb{E}[(\text{PB}(N,p)-z)^+]=\mathcal{A}g(k).\label{8:seqq}
\end{align}
Here, $f:\mathbb{Z}_+\to \mathbb{R}$ such that $f(k)=(k-z)^+$. Using the rv $\overline{W}_{\bar{n}}$ in place of $k$ and taking expectation, we get
\begin{align*}
\mathbb{E}[(\overline{W}_{\bar{n}}-z)^+]-\mathbb{E}[(\text{PB}(N,p)-z)^+]=\mathbb{E}[\mathcal{A}g(\overline{W}_{\bar{n}})].
\end{align*} 
Therefore, it is enough to deal with the right-hand side, that is, $\mathbb{E}[\mathcal{A}g(\overline{W}_{\bar{n}})]$.\\ 
Next, we move to find the upper bound for $\|\Delta g\|$. Following the steps similar to Lemma 1 of Neammanee and Yonghint \cite{NY}, for $z\ge 0$, we have
\begin{align}
\mathbb{E}[(\text{PB}(N,p)-z)^+]=\frac{1}{\delta}\sum_{m=1}^{\floor{N}}(m-z)^+\binom{N}{m}p^m q^{N-m}\le \frac{1}{\delta}\sum_{m=1}^{\floor{N}}m\binom{N}{m}p^m q^{N-m}= Np.\label{8:31}
\end{align}
It can be easily verified that \eqref{8:seqq} has a solution 
\begin{align*}
g(k)=-\sum_{j=k}^{\floor{N}}\frac{[N(N-1)\cdots (N-j+1)](k-1)!}{[N(N-1)\cdots (N-k+1)]j!}\left(\frac{p}{q}\right)^{j-k}\big((j-z)^+-\mathbb{E}[(\text{PB}(N,p)-z)^+]\big),
\end{align*}
for $k\ge 1$. For details, see (2.6) of Eichelsbacher and Reinert \cite{ER}. Now, following the steps similar to Lemma 2 Neammanee and Yonghint \cite{NY}, for $k\ge 1$, we get
\begin{align}
0<\sum_{j=k}^{\floor{N}}\frac{[N(N-1)\cdots (N-j+1)](k-1)!}{[N(N-1)\cdots (N-k+1)]j!}\left(\frac{p}{q}\right)^{j-k}(j-z)^+&\le 1+\sum_{j=1}^{\floor{N-k}}\binom{N-k}{j}\left(\frac{p}{q}\right)^j\nonumber\\
&\le q^{\ceil{k-N}}\le q^{\ceil{1-N}},\label{8:12}
\end{align}
where $\ceil{x}$ denote the least integer more than or equal to $x$. Also, for $k\le Np$, we have
\begin{align}
 \sum_{j=k}^{\floor{N}}\frac{[N(N-1)\cdots (N-j+1)](k-1)!}{[N(N-1)\cdots (N-k+1)]j!}\left(\frac{p}{q}\right)^{j-k}&\le \sum_{j=k}^{\floor{N}}\frac{(N-k)\cdots(N-j+1)}{(j-k+1)!}\left(\frac{p}{q}\right)^{j-k}\nonumber\\
&\le \frac{1}{N-k+1}\sum_{j=0}^{\floor{N-k}}\binom{N-k+1}{j+1}\left(\frac{p}{q}\right)^j\nonumber\\
&\le \frac{q^{\ceil{k-N}}-q}{(N-k)p}\le \frac{q^{-\ceil{N}}}{Np} \label{8:20}
\end{align}
and, for $k\ge np$,
\begin{align}
\sum_{j=k}^{\floor{N}}\frac{[N(N\hspace{-0.05cm}-\hspace{-0.05cm}1)\cdots (N\hspace{-0.05cm}-\hspace{-0.05cm}j\hspace{-0.05cm}+\hspace{-0.05cm}1)](k\hspace{-0.05cm}-\hspace{-0.05cm}1)!}{[N(N-1)\cdots (N-k+1)]j!}\left(\frac{p}{q}\right)^{j-k}\hspace{-0.15cm}&=\frac{1}{k}\left(\hspace{-0.1cm}1\hspace{-0.1cm}+ \hspace{-0.25cm}\sum_{j=k+1}^{\floor{N}}\frac{[N(N\hspace{-0.05cm}-\hspace{-0.05cm}1)\cdots (N\hspace{-0.05cm}-\hspace{-0.05cm}j\hspace{-0.05cm}+\hspace{-0.05cm}1)]k!}{[N(N\hspace{-0.05cm}-\hspace{-0.05cm}1)\cdots (N\hspace{-0.05cm}-\hspace{-0.05cm}k\hspace{-0.05cm}+\hspace{-0.05cm}1)]j!}\left(\frac{p}{q}\right)^{j-k}\right)\nonumber\\
&\le \frac{1}{k}\left(1+ \sum_{j=1}^{\floor{N-k}}\binom{N-k}{j}\left(\frac{p}{q}\right)^{j}\right)\le \frac{q^{-\ceil{N}}}{Np}.\label{8:21}
\end{align}
Combining \eqref{8:20} and \eqref{8:21}, for $k\ge 1$, we have
\begin{align}
0< \sum_{j=k}^{\floor{N}}\frac{[N(N-1)\cdots (N-j+1)](k-1)!}{[N(N-1)\cdots (N-k+1)]j!}\left(\frac{p}{q}\right)^{j-k} \le \frac{q^{-\ceil{N}}}{Np}.\label{8:32}
\end{align}
Therefore, from \eqref{8:31} and \eqref{8:32}, we have
\begin{align}
 \sum_{j=k}^{\floor{N}}\frac{[N(N-1)\cdots (N-j+1)](k-1)!}{[N(N-1)\cdots (N-k+1)]j!}\left(\frac{p}{q}\right)^{j-k}\mathbb{E}[(\text{PB}(N,p)-z)^+]\le q^{-\ceil{N}}.\label{8:13}
\end{align}
Next, observe that
\begin{align*}
\Delta g(k)=g(k+1)-g(k)=C_k+D_k,
\end{align*}
where
\vspace{-0.34cm}
\begin{align*}
C_k=&\sum_{j=k}^{\floor{N}}\frac{[N(N-1)\cdots (N-j+1)](k-1)!}{[N(N-1)\cdots (N-k+1)]j!}\left(\frac{p}{q}\right)^{j-k}(j-z)^+\\
&-\sum_{j=k+1}^{\floor{N}}\frac{[N(N-1)\cdots (N-j+1)]k!}{[N(N-1)\cdots (N-k)]j!}\left(\frac{p}{q}\right)^{j-k-1}(j-z)^+
\end{align*}
and
\vspace{-0.34cm}
\begin{align*}
D_k=&\sum_{j=k+1}^{\floor{N}}\frac{[N(N-1)\cdots (N-j+1)]k!}{[N(N-1)\cdots (N-k)]j!}\left(\frac{p}{q}\right)^{j-k-1}\mathbb{E}[(\text{PB}(N,p)-z)^+]\\
&-\sum_{j=k}^{\floor{N}}\frac{[N(N-1)\cdots (N-j+1)](k-1)!}{[N(N-1)\cdots (N-k+1)]j!}\left(\frac{p}{q}\right)^{j-k}\mathbb{E}[(\text{PB}(N,p)-z)^+].
\end{align*}
Using \eqref{8:12} and \eqref{8:13}, we have
\begin{align*}
|\Delta g|\le |C_k|+|D_k|\le q^{-\ceil{N}}(1+q).
\end{align*}
Hence, from \eqref{8:mmmd}, we have
\begin{align}
\left|\mathbb{E}[\mathscr{A}g(\overline{W}_{\bar{n}})]\right|&\le  \frac{(1+q)}{q^{\ceil{N}+1}}\left\{ \frac{1}{2}\left[\sum_{i=1}^{\bar{n}} \omega_i\mathbb{E}I_i \mathbb{E}[I_{A_i}^*(2I_{B_i}^{*}-I_{A_i}^{*}-1)\mathcal{D}(\overline{W}_{\bar{n}}|I_{A_i},I_{B_i})]\right.\right.\nonumber\\
&~~~~~~~~~~~~~~~~~~~~~~\left.+\sum_{i=1}^{\bar{n}}\omega_i\mathbb{E}[I_i(I_{A_i}^{*}-1)(2I_{B_i}^{*}-I_{A_i}^{*}-2)\mathcal{D}(\overline{W}_{\bar{n}}|I_i,I_{A_i},I_{B_i})]\right]\nonumber\\
&~~~~~~~~~~~~~~~~~~~~~~~+\sum_{i=1}^{\bar{n}}\omega_i\bigr|\mathbb{E}(I_i) \mathbb{E}(I_{A_i^*})-\mathbb{E}(I_i(I_{A_i}^{*}-1))-p\mathbb{E}(I_i)\bigr|\mathbb{E}[I_{B_i}^{*}\mathcal{D}(W|I_{B_i})]\nonumber\\
&~~~~~~~~~~~~~~~~~~~~~~~\left.+p\sum_{i=1}^{\bar{n}} \omega_i\mathbb{E}[I_i(I_{B_i}^{*}-1)\mathcal{D}(\overline{W}_{\bar{n}}|I_{B_i})]\right\},\label{8:15}
\end{align}
where $q=1-p=\mathrm{Var}(\overline{W}_{\bar{n}})/\mathbb{E}\overline{W}_{\bar{n}}$, $N=(\mathbb{E}\overline{W}_{\bar{n}})^2/(\mathbb{E}\overline{W}_{\bar{n}}-\mathrm{Var}(\overline{W}_{\bar{n}}))$, $I_{A}^{*}=\sum_{i\in A}\omega_i I_i$, $I_{A}$ is the collection of random variables $\{I_i:i\in A\}$, for $A\subset \{1,2,\ldots,\bar{n}\}$, and $\mathcal{D}(\overline{W}_{\bar{n}}|\cdot)$ can be computed subject to the exact structure of dependency. For example, if the dependency structure is the same as discussed in Example \ref{8:ex1} and $\omega_i=1$, $1\le i \le n$, then 
\begin{align*}
\left|\mathbb{E}[\mathscr{A}g(\overline{W}_{\bar{n}})]\right|&\le \frac{(1+q)}{q^{\ceil{N}+1}}\left\{\sum_{i=1}^{\bar{n}}\big[ \mathbb{E}I_i \mathbb{E}[I_{A_i}^*(2I_{B_i}^{*}-I_{A_i}^{*}-1)]+\mathbb{E}[I_i(I_{A_i}^{*}-1)(2I_{B_i}^{*}-I_{A_i}^{*}-2)]\big]\right.\\
&~~~~~~~~~~~~~~~~~~~~~~~+2\sum_{i=1}^{\bar{n}}\bigr|\mathbb{E}(I_i) \mathbb{E}(I_{A_i^*})-\mathbb{E}(I_i(I_{A_i}^{*}-1))-p\mathbb{E}(I_i)\bigr|\mathbb{E}[I_{B_i}^{*}]\\
&~~~~~~~~~~~~~~~~~~~~~~~\left.+2p \sum_{i=1}^{\bar{n}}\mathbb{E}[I_i(I_{B_i}^{*}-1)]\right\}\left(\frac{1}{2}\sum_{j \in \mathcal{F}}\min\left\{1,3(1-\mathbb{E}(I_{i-1}))\mathbb{E}I_i\right\}\right)^{-1/2},
\end{align*}
which is an improvement over the bound given in Theorem 2(1) of Yonghint et al. \cite{YNC2}.
\end{example}

\subsection{Independent Random Variables}
In this subsection, we consider $B_i=A_i=\{i\}$ in the earlier discussed setup and obtain approximation results for $W^*=\sum_{i\in J}\omega_i \eta_i$, the weighted sum of independent random variables. To simplify the presentation, let us define $p_i(k):=\mathbb{P}(\eta_i=k)$ and $\gamma:=2 \max_{i \in J}d_{TV}(W_i,W_i+1)$ where $W_i=W^*-\omega_i \eta_i$.

\newpage
\begin{theorem}\label{8:th2}
Let $W^*$ be the weighted sum of independent random variables and $X$ be a random variable having Stein operator \eqref{8:com} satisfying \eqref{8:MM}. Then
\begin{align}
d_{TV}(\mathcal{L}(W^*),\mathcal{L}(X))\le \gamma \|\Delta g\| \sum_{i\in J}\omega_i\left(\sum_{k=1}^{\infty}h_i(k)+d_i\right),\label{8:ibd}
\end{align}
where $d_i=\mathbb{E}(\omega_i\eta_i)|(1-\beta)[\mathbb{E}(\eta_i)^2-\mathbb{E}(\eta_i(\eta_i-1))]+\beta\mathbb{E}(\eta_i)|$ and
\begin{align*}
h_i(k)=\left\{\begin{array}{ll}
\frac{k(k-1)}{2}|(1-\beta)\mathbb{E}\eta_i p_i(k)+\beta k p_i(k)-(k+1)p_i(k+1)|, & \text{if $\omega_i=1$;}\\ 
\sum_{\ell=1}^{\omega_i k-1}|(1-\beta)\ell \mathbb{E}\eta_i+\beta\ell k-(\ell-1)k|p_i(k), & \text{if $\omega_i\ge 2$}.
\end{array}\right.
\end{align*}
\end{theorem}

\begin{remark}
Note that Remark 4.1 of Vellaisamy {et al.} \cite{VUC} can be applied to our results, and hence for $\gamma_j=\min\{1/2,1-d_{TV}(\omega_j\eta_j,\omega_j\eta_j+1)\}$, $\gamma^*=\max_{j \in J}\gamma_j$, we have
\begin{align*}
\gamma\le \sqrt{\frac{2}{\pi}}\left(\frac{1}{4}+\sum_{j\in J}\gamma_j-\gamma^*\right)^{-1/2}.
\end{align*} 
Therefore, if $\eta_J$ is of $O(n)$ then the bound \eqref{8:ibd} is of $O(n^{-1/2})$. Observe that the above bound for $\gamma$ is useful when $w_j=1$ for many values of $j$. 
\end{remark}

\begin{corollary}\label{8:cor1}
Assume the conditions of Theorem \ref{8:th2} hold with $X\sim$ PB$(N,p)$ and $\mathbb{E}W^*>\mathrm{Var}(W^*)$. Then 
\begin{align*}
d_{TV}(\mathcal{L}(W^*),\text{PB}(N,p))\le \frac{\gamma}{\floor{N} pq} \sum_{i\in J}\omega_i\left(\sum_{k=1}^{\infty}h_i(k)+d_i\right),
\end{align*}
where  $d_i=\mathbb{E}(\omega_i\eta_i)|\mathbb{E}(\eta_i)^2-\mathbb{E}(\eta_i(\eta_i-1))-p\mathbb{E}(\eta_i)|$ and
\begin{align*}
h_i(k)=\left\{\begin{array}{ll}
\frac{k(k-1)}{2}|\mathbb{E}\eta_i p_i(k)-pkp_i(k)-q(k+1)p_i(k+1)|, & \text{if $\omega_i=1$;}\\ 
\sum_{\ell=1}^{\omega_i k-1}|\ell \mathbb{E}\eta_i-p\ell k-q(\ell-1)k|p_i(k), & \text{if $\omega_i\ge 2$}.
\end{array}\right.
\end{align*}
\end{corollary}

\begin{remarks}
\begin{itemize}
\item[(i)] If $J=\{1,2,\ldots,n\}$, $\omega_i=1$, and $\eta_i\sim$ Ber$(p)$, for all $1\le i \le n$. Then, $h_i(k)=d_i=0$, and hence $d_{TV}(\mathcal{L}(W^*),\text{PB}(N,p))=0$, as expected.
\item[(ii)] Let $\omega_i=1$ and $\eta_i\sim$ Ber$(p_i)$, for $i\in J=\{1,2,\ldots,n\}$. Then, from Corollary \ref{8:cor1}, we have
\begin{align}
d_{TV}(\mathcal{L}(W^*),PB(N,p))\le\sqrt{\frac{2}{\pi}}\left(\frac{1}{4}+\sum_{i=1}^{n}\bar{\gamma}_i-\bar{\gamma}^*\right)^{-1/2}\frac{1}{\floor{N}pq}\sum_{i=1}^{n}p_i^2|p-p_i|,\label{8:www}
\end{align}
where $\bar{\gamma}_j=\frac{1}{2}\min\{1,1+p_i-|1-2p_i|\}$ and $\bar{\gamma}^*=\max_{1\le i\le n}\bar{\gamma}_i$. The bound given in \eqref{8:www} is of $O(n^{-1/2})$ and is an order improvement over Theorem 1 of Barbour and Hall \cite{BH}, Theorem 9.E of Barbour {et al.} \cite{BHJ}, and the bounds discussed by Kerstan \cite{KJ} and Le Cam \cite{CAM}.
\item[(iii)] Consider the setup of CDO discussed in Example \ref{8:ex2} under independent Bernoulli trials and unit weights, that is, $\overline{W}_n^*=\sum_{i=1}^{n}I_i$ where $I_i$, for $1\le i\le n$, are independent Bernoulli trials. Using $A_i=B_i=\{i\}$ in \eqref{8:mmmm}, routine calculations lead to 
\begin{align}
\big|\mathbb{E}[\mathscr{A}g(\overline{W}_n^*)]\big|\le \frac{(1+q)}{q^{\ceil{N}+1}}\sum_{i=1}^{n}p_i|p-p_i|,  \label{8:200}
\end{align}
where $p=\frac{1}{N}\sum_{i=1}^{n}p_i$. Note that if $p_i=p$ in \eqref{8:200}, for $1\le i\le n$, then $\big|\mathbb{E}[\mathscr{A}g(\overline{W}_n^*)]\big|=0$, as expected. Also, from \eqref{8:final} with $\omega_i=1$, for $1\le i \le n$, we have
\begin{align}
\big|\mathbb{E}[\mathscr{A}g(\overline{W}_n^*)]\big|\le \sqrt{\frac{2}{\pi}}\left(\frac{1}{4}+\sum_{i=1}^{n}\tilde{\gamma}_i-\tilde{\gamma}^*\right)^{-1/2}\frac{1+q}{q^{\ceil{N}+1}}\sum_{i=1}^{n}p_i^2|p-p_i|, \label{8:wwww}
\end{align}
where $\tilde{\gamma}_i=\frac{1}{2}\min\{1,1+p_i-|1-2p_i|\}$ and $\tilde{\gamma}^*=\max_{1\le i\le n}\bar{\gamma}_i$, $q=1-p=\sum_{i=1}^{n}p_iq_i/\sum_{i=1}^{n}p_i$, and $N=(\sum_{i=1}^{n}p_i)^2/\sum_{i=1}^{n}p_i^2$. For Poisson approximation, the existing bound given in (4) of Neammanee and Yonghint \cite{NY} is
\begin{align}
\big|\mathbb{E}[\mathscr{A}g(\overline{W}_n^*)]\big|\le \left(2\exp \left(\sum_{i=1}^{n}p_i\right)-1\right)\sum_{i=1}^{n}p_i^2.\label{8:300}
\end{align} 
Note that, for small values of $p_i$, the bound given in \eqref{8:200} is better than the bound given in \eqref{8:300}. For instance, let $n=50$ and $p_i$, $1\le i\le 50$, be defined as follows:
\begin{table}[H]
  \centering
  \begin{tabular}{cccccccc}
\toprule
$i$ & $p_i$ & $i$ & $p_i$ & $i$ & $p_i$\\
\midrule
0-10 & 0.05 & 21-30 & 0.15 &  41-50 & 0.25\\
11-20 & 0.10& 31-40 & 0.20 &\\
\bottomrule
\end{tabular}
\end{table}
Next, the following table gives a comparison between \eqref{8:200}, \eqref{8:wwww}, and \eqref{8:300}.
\begin{table}[H]
  \centering
  \begin{tabular}{llllllllll}
\toprule
$n$ & From \eqref{8:200} & From \eqref{8:wwww} & From \eqref{8:300}\\
\midrule
10  & $7.14 \times 10^{-17}$ & $6.00\times 10^{-18}$ & $0.0574$\\
20  & $0.3711$               & $0.01630$ & $0.9954$\\
30  & $4.9800$               & $0.36415$ & $13.7099$\\
40  & $111.8440$             & $11.7054$ & $221.8700$\\
50  & $3311.4600$            & $897.600$ & $4970.7400$\\
\bottomrule
\end{tabular}
\end{table}
Observe that our bounds are better than existing bounds for various values of $N$ and $p_i$. 
\end{itemize}
\end{remarks}

\begin{corollary}\label{8:cor2}
Assume the conditions of Theorem \ref{8:th2} hold with $X\sim$ NB$(r,\bar{p})$ and $\mathbb{E}W^*<\mathrm{Var}(W^*)$. Then 
\begin{align}
d_{TV}(\mathcal{L}(W^*),\text{NB}(r,\bar{p}))\le \frac{\gamma}{r \bar{q}} \sum_{i\in J}\omega_i\left(\sum_{k=1}^{\infty}h_i(k)+d_i\right),\label{8:100}
\end{align}
where $d_i=\mathbb{E}(\omega_i\eta_i)|\bar{p}[\mathbb{E}(\eta_i)^2-\mathbb{E}(\eta_i(\eta_i-1))]+\bar{q}\mathbb{E}(\eta_i)|$ and
\begin{align*}
h_i(k)=\left\{\begin{array}{ll}
\frac{k(k-1)}{2}|\bar{p}\mathbb{E}\eta_i p_i(k)+\bar{q} kp_i(k)-(k+1)p_i(k+1)|, & \text{if $\omega_i=1$;}\\ 
\sum_{\ell=1}^{\omega_i k-1}|\bar{p}\ell \mathbb{E}\eta_i+\bar{q}\ell k-(\ell-1)k|p_i(k), & \text{if $\omega_i\ge 2$}.
\end{array}\right.
\end{align*}
\end{corollary}

\begin{remarks}
\begin{itemize}
\item[(i)] If $J=\{1,2,\ldots,n\}$ and $\omega_i=1$, for all $i$, then, from Corollary \ref{8:cor2}, we have
\begin{align}
d_{TV}(\mathcal{L}(W^*),\text{NB}(r,\bar{p}))\le \frac{\gamma}{r \bar{q}}  \sum_{i=1}^{n}\left(\sum_{k=2}^{\infty}\frac{k(k\hspace{-0.05cm}-\hspace{-0.05cm}1)}{2}|\bar{p}\mathbb{E}\eta_i p_i(k)\hspace{-0.05cm}+\hspace{-0.05cm}\bar{q}kp_i(k)\hspace{-0.05cm}-\hspace{-0.05cm}(k\hspace{-0.05cm}+\hspace{-0.05cm}1)p_i(k\hspace{-0.05cm}+\hspace{-0.05cm}1)|+d_i\right)\hspace{-0.05cm},\label{8:nbbd}
\end{align}
which is an improvement over the bound given in $(17)$ of Vellaisamy {et al.} \cite{VUC}. Also, if $\eta_i\sim $ Geo$(\bar{p})$, the geometric distribution, for $1\le i\le n$, then $d_{TV}(\mathcal{L}(W^*),\text{NB}(n,\bar{p}))=0$, as expected.
\item[(ii)] If $\eta_i\sim \text{NB}(n_i,p_i)$, $1\le i \le n$, then the bound given in \eqref{8:nbbd} leads to
\begin{align*}
d_{TV}(\mathcal{L}(W^*),\text{NB}(r,\bar{p}))\le \frac{\gamma^*}{r \bar{q}}  \sum_{i=1}^{n}\left(\bar{p}(n_iq_i+1)\left|\frac{q_i}{p_i}-\frac{\bar{q}}{\bar{p}}\right|\frac{n_i(n_i+1)q_i^2}{p_i^2}+d_i\right),
\end{align*}
where $\gamma^*\le \sqrt{\frac{2}{\pi}}\left(\frac{1}{4}+\sum_{i=1}^{n}\mathbb{P}(\eta_i=\floor{(n_i-1)q_i / p_i})-\max_{1\le i\le n}\mathbb{P}(\eta_i=\floor{(n_i-1)q_i / p_i})\right)^{-1/2}$ and $q_i=1-p_i$, which is an order improvement over the bound given in Theorem 3.1 of Teerapabolarn \cite{Teeraa2015}.
\end{itemize}
\end{remarks}

\begin{example}[Compound Poisson Distribution]
Let $w_i=i$, $\eta_i\sim \text{Po}(\lambda_i)$, the Poisson distribution, for $i\in J=\{1,2,\ldots,n\}$, and $S_n=\sum_{i=1}^{n}i\eta_i$. The distribution of $S_\infty$ is known as compound Poisson distribution. The mean and variance of $S_n$ satisfy $\sum_{i=1}^{n}i\lambda_i=\mathbb{E}S_n< \mathrm{Var}(S_n)=\sum_{i=1}^{n}i^2\lambda_i$. Therefore, the negative binomial approximation is suitable in the view of the applicability of parameters. Hence, from \eqref{8:100}, we have
\begin{align*}
 d_{TV}(\mathcal{L}(S_n),\text{NB}(r,\bar{p}))\le\sqrt{\frac{2}{\pi}}\frac{2}{r \bar{q}} \sum_{i=1}^{n}E_i,
\end{align*}
where $\bar{q}=\sum_{i=1}^{n}i(i-1)\lambda_i\big/ \sum_{i=1}^{n}i^2\lambda_i$ and $r=\left(\sum_{i=1}^{n}i\lambda_i\right)^2\big/ \sum_{i=1}^{n}i(i-1)\lambda_i$, and
\begin{align*}
E_i\le\left\{\begin{array}{ll}
\bar{q}\lambda_i^2(\lambda_i+1)+\bar{q}(i\lambda_i)^2, & \text{if $i=1$;}\\ 
\frac{1}{2}\mathbb{E}[i\eta_i(i\eta_i-1)(\bar{p}(i\eta_i +i\lambda_i)+2)]+\bar{q}(i\lambda_i)^2, & \text{if $i\ge 2$}.
\end{array}\right.
\end{align*} 
Note that if $i\lambda_i$ is decreasing in $i$ then the bound is useful in practice. For similar conditions, see Barbour {et al.} \cite{BCL} and Gan and Xia \cite{GX}.
\end{example}

\section{Proofs}\label{8:pr}
In this section, we prove the main results presented in Section \ref{8:MR}.

\begin{proof}[Proof of Theorem \ref{8:th1}]
Consider the Stein operator \eqref{8:com} and taking expectation with respect to $W$, we have
\begin{align*}
\mathbb{E}[\mathscr{A}_Xg(W)]&=\alpha \mathbb{E}[g(W+1)]+\beta \mathbb{E}[Wg(W+1)]-\mathbb{E}[Wg(W)]\\
&=(1-\beta)\left[\frac{\alpha}{1-\beta} \mathbb{E}[g(W+1)]- \mathbb{E}[Wg(W)]\right]+\beta \mathbb{E}[W\Delta g(W)].
\end{align*}
Using \eqref{8:MM}, the above expression leads to
\begin{align}
\mathbb{E}[\mathscr{A}_Xg(W)]&=(1-\beta)\left[\sum_{i\in J}\omega_i \mathbb{E}\eta_i \mathbb{E}[g(W+1)]- \sum_{i\in J}\omega_i \mathbb{E}[\eta_ig(W)]\right]+\beta \mathbb{E}[W\Delta g(W)].\label{8:mmmm}
\end{align}
Let $W_i=W-\sum_{j\in A_i}\omega_i \eta_i=W-\eta_{A_i}^{*}$ so that $\eta_i$ and $W_i$ are independent random variables. Therefore,
\begin{align}
\mathbb{E}[\mathscr{A}_Xg(W)]&=(1-\beta)\sum_{i\in J}\omega_i \mathbb{E}\eta_i \mathbb{E}[g(W+1)-g(W_i+1)]\nonumber\\
&~~~- (1-\beta)\sum_{i\in J}\omega_i \mathbb{E}[\eta_i(g(W)-g(W_i+1))]+\beta \mathbb{E}[W\Delta g(W)]\nonumber\\
&=(1-\beta)\sum_{i\in J}\omega_i \mathbb{E}\eta_i \mathbb{E}\Bigg[\sum_{j=1}^{\eta_{A_i}^{*}}\Delta g(W_i+j)\Bigg]\nonumber\\
&~~~- (1-\beta)\sum_{i\in J}\omega_i \mathbb{E}\Bigg[\eta_i\sum_{j=1}^{\eta_{A_i}^{*}-1}\Delta g(W_i+j)\Bigg]+\beta \sum_{i\in J}\omega_i \mathbb{E}[\eta_i\Delta g(W)].\label{8:mm1}
\end{align}
Next, let $W_i^*=W-\sum_{j\in B_i}\omega_i \eta_i=W-\eta_{B_i}^{*}$ so that $\eta_i$ and $\eta_{A_i}$ are independent of $W_i^*$. Also, observe that
\begin{align}
&(1-\beta)\left\{\sum_{i\in J}\omega_i \mathbb{E}\eta_i \mathbb{E}\Bigg[\sum_{j=1}^{\eta_{A_i}^{*}}1\Bigg]- \sum_{i\in J}\omega_i \mathbb{E}\Bigg[\eta_i\sum_{j=1}^{\eta_{A_i}^{*}-1}1\Bigg]\right\}+\beta \sum_{i\in J}\omega_i \mathbb{E}\eta_i\label{8:mm2}\\
&=(1-\beta)\left\{\sum_{i\in J}\omega_i \mathbb{E}\eta_i \mathbb{E}[\eta_{A_i}^{*}]- \sum_{i\in J}\omega_i \mathbb{E}[\eta_i(\eta_{A_i}^{*}-1)]\right\}+\beta \sum_{i\in J}\omega_i \mathbb{E}\eta_i\nonumber\\
&=(1-\beta)\left[\frac{1}{1-\beta}\sum_{i\in J}\omega_i \mathbb{E}\eta_i-\sum_{i\in J}\sum_{j\in A_i}\omega_i\omega_j(\mathbb{E}[\eta_i \eta_j]-\mathbb{E}\eta_i \mathbb{E}\eta_j)\right]\nonumber\\
&=(1-\beta)\left[\frac{1}{1-\beta}\mathbb{E}(W_n)- \mathrm{Var}(W_n)\right]=0.\nonumber
\end{align}
Multiply $\mathbb{E}[\Delta g(W+1)]$ in \eqref{8:mm2} and using the corresponding expression in \eqref{8:mm1}, we get
\vspace{-0.38cm}
\begin{align}
\mathbb{E}[\mathscr{A}_Xg(W)]&=(1-\beta)\sum_{i\in J}\omega_i \mathbb{E}\eta_i \mathbb{E}\Bigg[\sum_{j=1}^{\eta_{A_i}^{*}}(\Delta g(W_i+j)-\Delta g(W_i^*+1))\Bigg]\nonumber\\
&~~~- (1-\beta)\sum_{i\in J}\omega_i \mathbb{E}\Bigg[\eta_i\sum_{j=1}^{\eta_{A_i}^{*}-1}(\Delta g(W_i+j)-\Delta g(W_i^*+1))\Bigg]\nonumber\\
&~~~+\beta \sum_{i\in J}\omega_i \mathbb{E}[\eta_i(\Delta g(W)-\Delta g(W_i^*+1))]\nonumber\\
&~~~-\hspace{-0.08cm}\sum_{i\in J}\omega_i\{(1\hspace{-0.08cm}-\hspace{-0.08cm}\beta)[\mathbb{E}(\eta_i) \mathbb{E}(\eta_{A_i}^{*})\hspace{-0.08cm}-\hspace{-0.08cm}\mathbb{E}(\eta_i(\eta_{A_i}^{*}\hspace{-0.08cm}-\hspace{-0.08cm}1))]\hspace{-0.08cm}+\hspace{-0.08cm}\beta\mathbb{E}(\eta_i)\}\mathbb{E}[\Delta g(W+1)\hspace{-0.08cm}-\hspace{-0.08cm}\Delta g(W_i^*+1)]\nonumber\\
&=(1-\beta)\sum_{i\in J}\omega_i \mathbb{E}\eta_i \mathbb{E}\Bigg[\sum_{j=1}^{\eta_{A_i}^{*}}\sum_{\ell=1}^{\eta_{B_i\backslash A_i}^*+j-1}\Delta^2 g(W_i^*+\ell)\Bigg]\nonumber\\
&~~~- (1-\beta)\sum_{i\in J}\omega_i \mathbb{E}\Bigg[\eta_i\sum_{j=1}^{\eta_{A_i}^{*}-1}\sum_{\ell=1}^{\eta_{B_i\backslash A_i}^*+j-1}\Delta^2 g(W_i^*+\ell)\Bigg]\nonumber\\
&~~~+\beta \sum_{i\in J}\omega_i \mathbb{E}\Bigg[\eta_i\sum_{\ell=1}^{\eta_{B_i}^{*}-1}\Delta^2 g(W_i^*+\ell)\Bigg]\nonumber\\
&~~~-\sum_{i\in J}\omega_i\{(1-\beta)[\mathbb{E}(\eta_i) \mathbb{E}(\eta_{A_i}^{*})-\mathbb{E}(\eta_i(\eta_{A_i}^{*}-1))]+\beta\mathbb{E}(\eta_i)\}\mathbb{E}\left[\sum_{\ell=1}^{\eta_{B_i}^{*}}\Delta^2 g(W_i^*+\ell)\right]\nonumber\\
&=(1-\beta)\sum_{i\in J}\omega_i \mathbb{E}\eta_i \mathbb{E}\Bigg[\sum_{j=1}^{\eta_{A_i}^{*}}\sum_{\ell=1}^{\eta_{B_i\backslash A_i}^*+j-1}\mathbb{E}[\Delta^2 g(W_i^*+\ell)|\eta_{A_i},\eta_{B_i}]\Bigg]\nonumber\\
&~~~- (1-\beta)\sum_{i\in J}\omega_i \mathbb{E}\Bigg[\eta_i\sum_{j=1}^{\eta_{A_i}^{*}-1}\sum_{\ell=1}^{\eta_{B_i\backslash A_i}^*+j-1}\mathbb{E}[\Delta^2 g(W_i^*+\ell)|\eta_i,\eta_{A_i},\eta_{B_i}]\Bigg]\nonumber\\
&~~~+\beta \sum_{i\in J}\omega_i \mathbb{E}\Bigg[\eta_i\sum_{\ell=1}^{\eta_{B_i}^{*}-1}\mathbb{E}[\Delta^2 g(W_i^*+\ell)|\eta_{B_i}]\Bigg]\nonumber\\
&~~~-\hspace{-0.08cm}\sum_{i\in J}\omega_i\{(1\hspace{-0.08cm}-\hspace{-0.08cm}\beta)[\mathbb{E}(\eta_i) \mathbb{E}(\eta_{A_i}^{*})\hspace{-0.08cm}-\hspace{-0.08cm}\mathbb{E}(\eta_i(\eta_{A_i}^{*}\hspace{-0.08cm}-\hspace{-0.08cm}1))]\hspace{-0.08cm}+\hspace{-0.08cm}\beta\mathbb{E}(\eta_i)\}\mathbb{E}\left[\sum_{\ell=1}^{\eta_{B_i}^{*}}\mathbb{E}[\Delta^2 g(W_i^*\hspace{-0.08cm}+\hspace{-0.08cm}\ell)|\eta_{B_i}]\right]. \label{8:mm3}
\end{align}
Hence,
\begin{align}
\left|\mathbb{E}[\mathscr{A}_Xg(W)]\right|&\le  \|\Delta g\|\left\{ \frac{(1-\beta)}{2}\left[\sum_{i\in J} \omega_i\mathbb{E}\eta_i \mathbb{E}[\eta_{A_i}^*(2\eta_{B_i}^{*}-\eta_{A_i}^{*}-1)\mathcal{D}(W|\eta_{A_i},\eta_{B_i})]\right.\right.\nonumber\\
&~~~~~~~~~~~~~~~~\left.+\sum_{i\in J}\omega_i\mathbb{E}[\eta_i(\eta_{A_i}^{*}-1)(2\eta_{B_i}^{*}-\eta_{A_i}^{*}-2)\mathcal{D}(W|\eta_i,\eta_{A_i},\eta_{B_i})]\right]\nonumber\\
&~~~~~~~~~~~~~~~~+\sum_{i\in J}\omega_i\bigr|(1-\beta)[\mathbb{E}(\eta_i) \mathbb{E}(\eta_{A_i}^{*})-\mathbb{E}(\eta_i(\eta_{A_i}^{*}-1))]+\beta\mathbb{E}(\eta_i)\bigr|\mathbb{E}[\eta_{B_i}^{*}\mathcal{D}(W|\eta_{B_i})]\nonumber\\
&~~~~~~~~~~~~~~~~+\left.|\beta| \sum_{i\in J} \omega_i\mathbb{E}[\eta_i(\eta_{B_i}^{*}-1)\mathcal{D}(W|\eta_{B_i})]\right\}.\label{8:mmmd}
\end{align}
Using \eqref{8:power}, the proof follows.
\end{proof}

\begin{proof}[Proof of Theorem \ref{8:th2}]
Let $A_i=B_i=\{i\}$ then $\{\eta_i, i\in J\}$ becomes independent random variables, and $W_i=W_i^*$ is independent of $\eta_{A_i}=\eta_{B_i}=\eta_i$. Therefore, from \eqref{8:mm3}, we have
\begin{align*}
\mathbb{E}[\mathscr{A}_Xg(W^*)]&=(1-\beta)\sum_{i\in J}\omega_i \mathbb{E}\eta_i \mathbb{E}\Bigg[\sum_{j=1}^{\omega_i \eta_i}\sum_{\ell=1}^{j-1}\mathbb{E}[\Delta^2 g(W_i+\ell)]\Bigg]\\
&~~~- (1-\beta)\sum_{i\in J}\omega_i \mathbb{E}\Bigg[\eta_i\sum_{j=1}^{\omega_i \eta_i-1}\sum_{\ell=1}^{j-1}\mathbb{E}[\Delta^2 g(W_i+\ell)]\Bigg]\\
&~~~+\beta \sum_{i\in J}\omega_i \mathbb{E}\Bigg[\eta_i\sum_{\ell=1}^{\omega_i \eta_i-1}\mathbb{E}[\Delta^2 g(W_i+\ell)]\Bigg]\\
&~~~-\sum_{i\in J}\omega_i\{(1-\beta)[\mathbb{E}(\eta_i)^2-\mathbb{E}(\eta_i(\eta_{i}-1))]+\beta\mathbb{E}(\eta_i)\}\mathbb{E}\left[\sum_{\ell=1}^{\omega_i\eta_i}\mathbb{E}[\Delta^2 g(W_i+\ell)]\right]\\
&=(1-\beta)\sum_{i\in J}\omega_i\mathbb{E}\eta_i \sum_{k=1}^{\infty}\sum_{j=1}^{\omega_i k}\sum_{\ell=1}^{j-1}\mathbb{E}[\Delta^2 g(W_i+\ell)]p_i(k)\\
&~~~-(1-\beta)\sum_{i\in J}\omega_i \sum_{k=1}^{\infty}\sum_{j=1}^{\omega_i k-1}\sum_{\ell=1}^{j-1}k\mathbb{E}[\Delta^2 g(W_i+\ell)]p_i(k)\\
&~~~+\beta\sum_{i\in J}^{n}\omega_i \sum_{k=1}^{\infty}\sum_{\ell=1}^{\omega_i k-1}k \mathbb{E}[\Delta^2 g(W_i+\ell)]p_i(k)\\
&~~~-\sum_{i\in J}\omega_i\{(1-\beta)[\mathbb{E}(\eta_i)^2-\mathbb{E}(\eta_i(\eta_{i}-1))]+\beta\mathbb{E}(\eta_i)\}\sum_{k=1}^{\infty}\sum_{\ell=1}^{\omega_i k}\mathbb{E}[\Delta^2 g(W_i+\ell)]p_i(k)\\
&=\sum_{i\in J}\omega_i \sum_{k=1}^{\infty}\sum_{j=1}^{\omega_i k}\sum_{\ell=1}^{j-1}[(1-\beta)\mathbb{E}\eta_i+\beta k]\mathbb{E}[\Delta^2 g(W_i+\ell)]p_i(k)\\
&~~~-\sum_{i\in J}\omega_i \sum_{k=1}^{\infty}\sum_{j=1}^{\omega_i k-1}\sum_{\ell=1}^{j-1}k\mathbb{E}[\Delta^2 g(W_i+\ell)]p_i(k)\\
&~~~-\sum_{i\in J}\omega_i\{(1-\beta)[\mathbb{E}(\eta_i)^2-\mathbb{E}(\eta_i(\eta_{i}-1))]+\beta\mathbb{E}(\eta_i)\}\sum_{k=1}^{\infty}\sum_{\ell=1}^{\omega_i k}\mathbb{E}[\Delta^2 g(W_i+\ell)]p_i(k)\\
&=\sum_{i\in J}\omega_i \sum_{k=1}^{\infty}\sum_{\ell=1}^{\omega_i k-1}\ell [(1-\beta)\mathbb{E}\eta_i+\beta k] \mathbb{E}[\Delta^2 g(W_i+\omega_i k-\ell)]p_i(k)\\
&~~~-\sum_{i\in J}\omega_i \sum_{k=1}^{\infty}\sum_{\ell=1}^{\omega_i k-2}\ell k \mathbb{E}[\Delta^2 g(W_i+\omega_i k-\ell-1)]p_i(k)\\
&~~~-\sum_{i\in J}\omega_i\{(1-\beta)[\mathbb{E}(\eta_i)^2-\mathbb{E}(\eta_i(\eta_{i}-1))]+\beta\mathbb{E}(\eta_i)\}\sum_{k=1}^{\infty}\sum_{\ell=1}^{\omega_i k}\mathbb{E}[\Delta^2 g(W_i+\ell)]p_i(k).
\end{align*}
{\bf Case I}: If $\omega_i=1$ then
\begin{align}
\mathbb{E}[\mathscr{A}_Xg(W^*)]&=\sum_{i\in J}\sum_{k=2}^{\infty}\sum_{\ell=1}^{k-1}\ell [(1-\beta)\mathbb{E}\eta_i +\beta k]\mathbb{E}[\Delta^2 g(W_i+ k-\ell)]p_i(k)\nonumber\\
&~~~-\sum_{i\in J} \sum_{k=3}^{\infty}\sum_{\ell=1}^{k-2}\ell k \mathbb{E}[\Delta^2 g(W_i+k-\ell-1)]p_i(k)\nonumber\\
&~~~-\sum_{i\in J}\{(1-\beta)[\mathbb{E}(\eta_i)^2-\mathbb{E}(\eta_i(\eta_{i}-1))]+\beta\mathbb{E}(\eta_i)\}\sum_{k=1}^{\infty}\sum_{\ell=1}^{k}\mathbb{E}[\Delta^2 g(W_i+\ell)]p_i(k)\nonumber\\
&=\sum_{i\in J}\sum_{k=2}^{\infty}\sum_{\ell=1}^{k-1}\ell[(1-\beta)\mathbb{E}\eta_i p_i(k)+\beta kp_i(k)-(k+1)p_i(k+1)]\mathbb{E}[\Delta^2 g(W_i+ k-\ell)]\nonumber\\
&~~~-\sum_{i\in J}\{(1-\beta)[\mathbb{E}(\eta_i)^2-\mathbb{E}(\eta_i(\eta_{i}-1))]+\beta\mathbb{E}(\eta_i)\}\sum_{k=1}^{\infty}\sum_{\ell=1}^{k}\mathbb{E}[\Delta^2 g(W_i+\ell)]p_i(k).\label{8:mm4}
\end{align}
{\bf Case II}: If $\omega_i \ge 2$ then
\begin{align}
\mathbb{E}[\mathscr{A}g(W)]&=\sum_{i\in J}\omega_i \sum_{k=1}^{\infty}\sum_{\ell=1}^{\omega_i k-1}\ell [(1-\beta)\mathbb{E}\eta_i+\beta k] \mathbb{E}[\Delta^2 g(W_i+\omega_i k-\ell)]p_i(k)\nonumber\\
&~~~-\sum_{i\in J}\omega_i \sum_{k=1}^{\infty}\sum_{\ell=1}^{\omega_i k-2}\ell k \mathbb{E}[\Delta^2 g(W_i+\omega_i k-\ell-1)]p_i(k)\nonumber\\
&~~~-\sum_{i\in J}\omega_i\{(1-\beta)[\mathbb{E}(\eta_i)^2-\mathbb{E}(\eta_i(\eta_{i}-1))]+\beta\mathbb{E}(\eta_i)\}\sum_{k=1}^{\infty}\sum_{\ell=1}^{\omega_i k}\mathbb{E}[\Delta^2 g(W_i+\ell)]p_i(k)\nonumber\\
&=\sum_{i\in J}\omega_i \sum_{k=1}^{\infty}\sum_{\ell=1}^{\omega_i k-1}[(1-\beta )\ell \mathbb{E}\eta_i+\beta \ell k-(\ell-1)k]\mathbb{E}[\Delta^2 g(W_i+\omega_i k-\ell)]p_i(k)\nonumber\\
&~~~-\sum_{i\in J}\omega_i\{(1-\beta)[\mathbb{E}(\eta_i)^2-\mathbb{E}(\eta_i(\eta_{i}-1))]+\beta\mathbb{E}(\eta_i)\}\sum_{k=1}^{\infty}\sum_{\ell=1}^{\omega_i k}\mathbb{E}[\Delta^2 g(W_i+\ell)]p_i(k).\label{8:mm5}
\end{align}
It is shown that in Barbour and Xia \cite{BX} (see also Barbour and \v{C}ekanavi\v{c}ius \cite{BC}, p. 517) $|\mathbb{E}(\Delta^2 g(W_i+\cdot))|\le \gamma \|\Delta g\|$. Hence, from \eqref{8:mm4} and \eqref{8:mm5}, we have
\begin{align}
\left|\mathbb{E}[\mathscr{A}g(W)]\right|\le \gamma \|\Delta g\| \sum_{i\in J}\omega_i \left(\sum_{k=1}^{\infty}h_i(k)+d_i\right).\label{8:final}
\end{align}
Using \eqref{8:power}, the proof follows.
\end{proof}

\section*{Acknowledgement}
The author would like to thank Dr. Neelesh S. Upadhye for helpful comments and suggestions. The author is grateful to the associate editor and reviewers for many valuable suggestions, critical comments that improved the presentation of the paper.

\singlespacing
\bibliographystyle{PV}
\bibliography{PA2PSDBib}

\end{document}